\documentclass[a4paper]{amsart}

\usepackage[utf8]{inputenc}
\usepackage[english]{babel}
\usepackage[T1]{fontenc}
\usepackage{amsmath}
\usepackage{amsthm}
\usepackage{amssymb}
\usepackage{xcolor}
\usepackage{forest}
\usepackage{tikz}
\usetikzlibrary{arrows,automata,positioning}
\usepackage{mathtools}

\usepackage{enumerate}
\usepackage{float}

\usepackage{hyperref}
\theoremstyle{definition}
\newtheorem{defi}{Definition}[section]
\newtheorem{example}[defi]{Example}

\theoremstyle{plain}
\newtheorem{theorem}[defi]{Theorem}

\newtheorem{lemma}[defi]{Lemma}
\newtheorem{cor}[defi]{Corollary}

\begin{document}

\title{Upper bounds on the average height of random binary trees}

\author{Louisa Seelbach Benkner}
\email{louisa.seelbach@gmail.com}

\address{Universit{\"a}t Siegen, Germany}

\begin{abstract} 
We study the average height of random trees generated by leaf-centric binary tree sources as introduced by Zhang,
Yang and Kieffer in \cite{ZhangYK14}. 
A leaf-centric binary tree source induces for every
$n \geq 2$ a probability distribution on the set of binary trees with $n$ leaves. 
Our results generalize a result by Devroye, according to which the average height of a random binary search tree of size $n$ is in $\mathcal{O}(\log n)$.
\end{abstract}
\maketitle
\section{Introduction}

The \emph{height} of a rooted tree is the number of edges in the longest path from the root node to another node of the tree. Analyzing the height of random trees is a natural research topic not only in combinatorics, but also in theoretical computer science, as the height of a tree occurs, e.g., as the stack size used by some tree traversal algorithms, see, e.g.,~\cite{FlajoletOdlyzko82}.
The height of random trees has been studied extensively in various papers and with respect to various random tree models, e.g., for uniformly random plane trees~\cite{deBruijnKnuthRice72}, uniformly random unordered binary trees~\cite{BroutinFlajolet12}, uniformly random labelled unordered trees~\cite{renyiszekeres67}, and digital search trees and tries~\cite{Devroye84,Flajolet83,KnesslS00,Szpankowski91}, only to name a few.

In this work, we focus on random binary trees. Binary trees will be rooted ordered trees where every node is either a leaf or has exactly two children (a left child and a right child). 
We define the size of a binary tree as the number of leaves (as every binary tree with $n$ leaves has exactly $2n-1$ nodes in total).

So far, the expected height of random binary trees has mainly been studied with respect to uniformly random (ordered) binary trees and binary search trees.
A uniformly random ordered binary tree of size $n$ is a random tree whose corresponding probability distribution is the uniform probability distribution on the set of ordered binary trees of size $n$.
Flajolet and Odlyzko showed in \cite{FlajoletOdlyzko82} (in a more general setting) that the expected height of a uniformly random binary tree of size $n$ asymptotically grows as $2\sqrt{\pi n}(1+o(1))$.

Another important type of random trees are random binary search trees. A random binary search tree of size $n$ is a binary search tree built by inserting the keys $\{1, \dots, n\}$ according to a uniformly chosen random permutation on $\{1, \dots, n\}$. Random binary search trees naturally arise in theoretical computer science, see e.g. \cite{Drmota09}. 
In \cite{Devroye86}, Devroye proved that the expected height of a random binary search tree of size $n$ asymptotically grows as $c\ln(n)(1+o(1))$, where $c\approx 4.31107...$ is the unique solution of $c\ln(2e/c)=1$ with $c \geq 2$.
Previous results on the average height of random binary search trees were shown by Robson~\cite{Robson79,Robson82} and Pittel~\cite{Pittel84}, and further refined results were shown e.g.~with respect to the variance~\cite{Devroye95, Drmota02} (see also \cite{BroutinDevroyeMcleishSalle08, Devroye87, Devroye90, Drmota01,Drmota03, Drmota04, KnesslS02} for further refinements and generalisations).

The goal of this work is to obtain general upper bounds on the average height of random binary trees that cover large classes of distributions. 
A very general concept to model probability distributions on the set of binary trees of size $n$ was presented by Kieffer et al.~in \cite{KiefferYS09} (see also \cite{ZhangYK14}), where the authors introduce \emph{leaf-centric binary tree sources} as an extension of the classical notion of an information source on finite sequences to binary trees.

A leaf-centric binary tree source is induced by a mapping $\sigma: \mathbb{N}\times\mathbb{N}\to[0,1]$ that satisfies $\sum_{i=1}^{n-1}\sigma(i,n-i)=1$ for every $n \geq 2$. In other words, $\sigma$ restricted to the set of ordered pairs $S_n:=\{(i,n-i) \mid 1 \leq i \leq n-1\} $ is a probability mass function for every $n \geq 2$. 
A leaf-centric binary tree source randomly generates a binary tree of size $n$ as follows: we start with a single root node labeled with $n$ and randomly choose a pair $(i,n-i)$ according to the distribution $\sigma$ on $S_n$. Then, a left (resp. right) child labeled with $i$ (resp. $n-i$) is attached to the root, and the process is repeated with these two nodes. The process stops at nodes with label $1$ (and in the end, the node labels are omitted). In particular, such a mapping $\sigma$ induces a function $P_{\sigma}$ which restricts to a probability mass function on the set of binary trees of size $n$ for every $n \geq 1$.

The binary search tree model is the leaf-centric binary tree source where the mapping $\sigma$ corresponds to the uniform probability distribution on the set $S_n$ for every $n \geq 2$ (Example~\ref{ex:bst}). Moreover, the uniform probability distribution on the set of binary trees of size $n$ can be modeled as a leaf-centric binary tree source as well (Example~\ref{ex:uniform}). Another well-known leaf-centric tree source is the binomial random tree model \cite{KiefferYS09} (see Example~\ref{ex:dst}), where the mapping $\sigma$ corresponds to a binomial distribution on $S_n$. 

Recently, in \cite{SeelbachLohreyWagner18} (see also \cite{GanardiHuckeLohreySeelbach19,SeelbachLohrey18}), in the context of DAG-compression, the average number of distinct fringe subtrees in random trees was analysed with respect to several classes of leaf-centric binary tree sources.
More in detail, let $F_{n,\sigma}$ denote the (random) number of distinct fringe subtrees occurring in a random tree of size $n$ generated by the leaf-centric tree source corresponding to the mapping $\sigma$.
In \cite{SeelbachLohreyWagner18, GanardiHuckeLohreySeelbach19}, the authors present several classes of leaf-centric binary tree sources, each defined in terms of some conditions on the mapping $\sigma$, and then derived upper and lower bounds on the average $\mathbb{E}(F_{n,\sigma})$ that hold for all mappings $\sigma$ in the respective class.
In particular, results from~\cite{SeelbachLohreyWagner18, GanardiHuckeLohreySeelbach19} generalize previous results on the number of distinct fringe subtrees in random trees, in the sense that from these general theorems that cover whole classes of sources previously shown asymptotic bounds on the average number of fringe subtrees in random binary search trees and uniformly random binary trees could be deduced.

The main goal of this work is to carry out investigations similar to those of~\cite{SeelbachLohreyWagner18, GanardiHuckeLohreySeelbach19} with respect to another tree parameter: the average height $\mathbb{E}(H_{n,\sigma})$.
Specifically, we consider two classes of leaf-centric binary tree sources and obtain upper bounds on $\mathbb{E}(H_{n,\sigma})$ for mappings $\sigma$ from the respective classes.
Here and in the rest of the paper, increasing/decreasing functions are not necessarily strictly monotone: if $f$ is increasing, then $f(x) \geq f(y)$ whenever $x > y$, but not necessarily $f(x) > f(y)$ (and analogously for decreasing functions). The classes of leaf-centric tree sources considered in this work are the following.
\begin{itemize}
\item[(i)] A leaf-centric tree source with mapping $\sigma$ is called \emph{$\psi$-upper-bounded} for a decreasing function $\psi: \mathbb{R} \to (0,1]$ if there is a constant $N$ such that
for all $1 \leq i \leq n-1$ and $n \geq N$: 
\begin{align*}
\sigma(i,n-i)\leq \psi(n).
\end{align*}
\item[(ii)] A leaf-centric tree source with mapping $\sigma$ is called \emph{$\phi$-weakly-balanced} for a decreasing function $\phi: \mathbb{N} \to (0,1]$ if there is a constant $\gamma \in (0,\frac{1}{2})$ and an integer $N$ such that for all $n \geq N$:
\begin{align*}
\sum_{\gamma n \leq i \leq (1-\gamma)n}\sigma(i,n-i) \geq \phi(n).
\end{align*}
\end{itemize}
Both properties have been introduced in \cite{SeelbachLohreyWagner18, SeelbachLohrey18}.
Property (i) quantifies how close $\sigma$ is to the binary search tree model. The latter is $\psi$-upper-bounded for $\psi(n) = 1/(n-1)$, which is the smallest function $\psi$  for which (i) makes
sense (since  the sum over all $\sigma(i,n-i)$ for $1 \leq i \leq n-1$ has to be one).
Property (ii) generalizes the concept of balanced binary tree sources from \cite{GanardiHuckeLohreySeelbach19,ZhangYK14}: When randomly constructing a binary tree with respect to a leaf-centric tree source of
type (ii), the probability that the current weight is roughly equally split among the two children is bounded from below by the function $\phi$. Therefore, for slowly decreasing
functions $\phi$, balanced trees are preferred by this model.

As our main results, we obtain upper bounds on the expected height $\mathbb{E}(H_{n,\sigma})$ for the classes of $\psi$-upper-bounded sources and $\phi$-weakly-balanced sources (Theorem~\ref{thm:upper-bounded} and Theorem~\ref{thm:weakly-balanced-upper-bound}).
In particular, these two general main theorems both yield an upper bound of order $\mathcal{O}(\log n)$ for the average height in the binary search tree model (as shown in~\cite{Devroye86}). 
Moreover, from Theorem~\ref{thm:weakly-balanced-upper-bound}, we obtain an upper bound of order $\mathcal{O}(\sqrt{n}\log(n)^2)$ for the uniform distribution (which is slightly weaker than the upper bound shown in~\cite{FlajoletOdlyzko82}) and an upper bound of order $\mathcal{O}(\log n)$ for the binomial random tree model.
What sets our main theorems apart from previous results on the average height of random trees is that they cover whole classes of sources simultaneously.

\section{Preliminaries}\label{sec:preliminaries}

We use the classical Landau notations $\mathcal{O}$, $o$, $\Omega$ and $\omega$.
In the following, $\log x$ will denote the binary logarithm $\log_2 x$ of a positive real number $x$. With $e$ we denote Euler's number and the natural logarithm $\log_e x$ will be denoted by $\ln x$.
Let $\mathbb{N}$ denote the natural numbers without zero. 

\subsection{Binary Trees}

We write $\mathcal{T}$ for the set of \emph{(full) binary trees}, i.e., of rooted ordered trees such that each node has either exactly two or zero children.
The \emph{size} of a binary tree $t$ is the number of leaves of $t$ and denoted by $|t|$. 
With $\mathcal{T}_n$ we denote the set of binary trees which have exactly $n$ leaves.

A \emph{fringe subtree} of a tree is a subtree that consists of a node and all its descendants.
For a node $v$ of a binary tree $t \in \mathcal{T}$, let $t[v]$ denote the fringe subtree of $t$ which is rooted at $v$. 
For an inner node $v$ of $t$, we denote with $t_{\ell}[v]$ (resp.~$t_r[v]$)  the fringe subtree rooted at $v$'s left (resp.~right) child node. 
If $v$ is the root node of $t$, we write $t_{\ell}$ and $t_r$ for $t_{\ell}[v]$ and $t_r[v]$ and call $t_{\ell}$ (resp.~$t_r$) the \emph{left subtree} (resp. \emph{right subtree}) of $t$.
For a binary tree $t \in \mathcal{T}$, let $h(t)$ denote the \emph{height} of $t$, which is inductively defined as 
\begin{align*}
h(t)=\begin{cases}
0 \quad&\text{if } |t|=1,\\
1+\max(h(t_{\ell}),h(t_r)) &\text{otherwise.}
\end{cases}
\end{align*}

\subsection{Leaf-centric binary tree sources}\label{leafcentric}

With $\mathcal{L}$ we denote the set of all functions $\sigma \colon \mathbb{N} \times \mathbb{N} \rightarrow [0,1]$ which satisfy 
\begin{align}\label{sigmacond}
\sum_{i,j \geq 1 \atop i+j=k} \sigma(i,j) = 1 
\end{align}  
for every integer $k \geq 2$. A mapping $\sigma\in \mathcal{L}$ induces a probability mass function $P_{\sigma}\colon \mathcal{T}_n \rightarrow [0,1]$ for every $n \geq 1$ in the following way. 
Define $P_{\sigma} \colon \mathcal{T} \rightarrow [0,1]$ inductively by 
\begin{equation} \label{eq-prob-mass-function-P}
P_{\sigma}(t)=\begin{cases}
1 \quad &\text{if } |t| =1,\\
\sigma(|t_{\ell}|,|t_r|) \cdot P_{\sigma}(t_{\ell}) \cdot P_{\sigma}(t_r) \quad &\text{otherwise.}
\end{cases}
\end{equation}
A tuple $(\mathcal{T}, (\mathcal{T}_n)_{n \in \mathbb{N}}, P_{\sigma})$ with $\sigma \in \mathcal{L}$ is called a \emph{leaf-centric tree source} \cite{ZhangYK14, KiefferYS09}.
Intuitively, a leaf-centric binary tree source randomly generates a binary tree of size $n$ as follows: we start at the root node and determine the sizes of the left and of the right subtree, where the probability that the left subtree is of size $i$ for $i \in \{1,\dots, n-1\}$ (and consequently, the right subtree is of size $n-i$) is given by $\sigma(i,n-i)$. This process then recursively continues in the left and right subtree, i.e., the leaf-centric binary tree source then randomly generates a binary tree of size $i$ as the left subtree and a binary tree of size $n-i$ as the right subtree.
In particular, for $t \in \mathcal{T}$, we have
\begin{align*}
P_{\sigma}(t)=\!\!\!\!\prod_{v \text{ inner node of } t}\sigma(|t_{\ell}[v]|,|t_r[v]|).
\end{align*}

Several well-known random tree models can be characterized as leaf-centric tree sources.

\begin{example}\label{ex:bst}
The \emph{binary search tree model} \cite{Drmota09} corresponds to the leaf-centric binary tree source defined by 
\begin{align*}
\sigma_{bst}(i,n-i)=\frac{1}{n-1}
\end{align*}
for every $n \geq 2$ and $1 \leq i \leq n-1$. This distribution over binary trees arises if a binary search tree of size
$n$ is built by inserting the keys $1,\dots,n$ according to a uniformly chosen random permutation on $1,\dots, n$ \cite{Drmota09}.
\end{example}

\begin{example}\label{ex:uniform}
The \emph{uniform distribution} on the set of binary trees $\mathcal{T}_n$ is induced by the leaf-centric binary tree source with 
\begin{align*}
\sigma_{uni}(i,n-i)=\frac{|\mathcal{T}_i||\mathcal{T}_{n-i}|}{|\mathcal{T}_n|}
\end{align*}
for every $n \geq 2$ and $1 \leq i \leq n-1$ \cite{KiefferYS09}.
\end{example}

\begin{example}\label{ex:dst}
Fix a constant $p \in (0,1)$ and define
\begin{equation} \label{eq-dst}
\sigma_{bin,p}(i,n-i)=p^{i-1}(1-p)^{n-i-1}\binom{n-2}{i-1}
\end{equation}
for every $n\geq 2$ and $1 \leq i \leq n-1$. This leaf-centric binary tree source corresponds to the \emph{binomial random tree model}, which was studied in \cite{KiefferYS09} for the case $p=1/2$, and which is a slight variant of the \emph{digital search tree model}, \cite{Drmota09}.
\end{example}

In this paper, we are interested in the expected height of random binary trees.
With $T_{n,\sigma}$ we denote a random tree of size $n$ drawn from the set $\mathcal{T}_n$ according to the probability mass function $P_{\sigma}\colon \mathcal{T}_n \to [0,1]$ from
\eqref{eq-prob-mass-function-P}, and with $H_{n,\sigma}$ we denote the height of the random tree $T_{n, \sigma}$. 
In the following, we investigate $H_{n,\sigma}$ under certain conditions on the mapping $\sigma$. 
That is, we will assume that the mapping $\sigma$ satisfies certain properties and then derive bounds on $\mathbb{E}(H_{n,\sigma})$.

\section{Upper bounds on the average height}

In this section, we present two classes of leaf-centric tree sources, for which we will be able to obtain upper bounds on $\mathbb{E}(H_{n,\sigma})$.
The main technique we will apply in order to obtain upper bounds on $\mathbb{E}(H_{n,\sigma})$ for these two classes is the following.
Let $(\varphi_n)_{n \in \mathbb{N}}$ denote a decreasing sequence of real numbers, with $\varphi_n>1$ for every $n \in \mathbb{N}$.
In order to obtain an upper bound on $\mathbb{E}(H_{n,\sigma})$, we inductively bound 
\begin{align*}
\mathbb{E}(\varphi_n^{H_{n,\sigma}})=\sum_{t \in \mathcal{T}_n}P_{\sigma}(t)\varphi_n^{h(t)}
\end{align*}
for every $n \in \mathbb{N}$. 
The result for $\mathbb{E}(H_{n,\sigma})$ then follows by applying Jensen's inequality.
The following recursive estimate for $\mathbb{E}(\varphi_n^{H_{n,\sigma}})$ will be useful:

\begin{lemma}\label{lemma:recursive}
Let $(\varphi_n)_{n \in \mathbb{N}}$ be a decreasing sequence of real numbers, with $\varphi_n>1$ for every $n \in \mathbb{N}$. 
Then, for every $n \geq 2$, 
\begin{align*}
\mathbb{E}(\varphi_n^{H_{n,\sigma}})\leq \varphi_n\sum_{k=1}^{n-1}\sigma(k,n-k)\left(\mathbb{E}(\varphi_k^{H_{k,\sigma}})+\mathbb{E}(\varphi_{n-k}^{H_{n-k,\sigma}})\right).
\end{align*}
\end{lemma}
\begin{proof}
By the recursive definition of $P_\sigma$, equation \eqref{eq-prob-mass-function-P}, and the recursive definition of $h(t)$, we obtain:
\begin{align*}
\mathbb{E}(\varphi_n^{H_{n,\sigma}})&=\sum_{t \in \mathcal{T}_n}P_{\sigma}(t)\varphi_n^{h(t)}\\
&=\sum_{k=1}^{n-1}\sigma(k,n-k)\sum_{u \in \mathcal{T}_k}\sum_{v \in \mathcal{T}_{n-k}}P_{\sigma}(u)P_{\sigma}(v)\varphi_n^{1+\max(h(u),h(v))}\\
&\leq \varphi_n\sum_{k=1}^{n-1}\sigma(k,n-k)\sum_{u \in \mathcal{T}_k}\sum_{v \in \mathcal{T}_{n-k}}P_{\sigma}(u)P_{\sigma}(v)\left(\varphi_n^{h(u)}+\varphi_n^{h(v)}\right)\\
&\overset{(*)}{\leq} \varphi_n\sum_{k=1}^{n-1}\sigma(k,n-k)\sum_{u \in \mathcal{T}_k}\sum_{v \in \mathcal{T}_{n-k}}P_{\sigma}(u)P_{\sigma}(v)\left(\varphi_k^{h(u)}+\varphi_{n-k}^{h(v)}\right)\\
&=\varphi_n\sum_{k=1}^{n-1}\sigma(k,n-k)\left(\mathbb{E}(\varphi_k^{H_{k,\sigma}})+\mathbb{E}(\varphi_{n-k}^{H_{n-k,\sigma}})\right),
\end{align*}
where the inequality $(*)$ follows as the sequence $(\varphi_n)_{n \in \mathbb{N}}$ is decreasing.
 \end{proof}

\subsection{Upper-bounded sources}
The first natural class of leaf-centric tree sources we consider is the following, where the function $\sigma$ is bounded from above in terms of a function $\psi$.
Similar classes of leaf-centric tree sources were introduced in \cite{SeelbachLohreyWagner18} in the context of DAG-compression of random trees.

\begin{defi}[$\psi$-upper-bounded sources]
Let $\psi\colon \mathbb{N} \to (0,1]$ denote a decreasing function and let $N \geq 1$.
With $\mathcal{L}_{up}(\psi,N) \subseteq \mathcal{L}$ we denote the set of mappings $\sigma \in \mathcal{L}$, which satisfy
\begin{align*}
\sigma(i,n-i) \leq \psi(n)
\end{align*} 
for every integer $1 \leq i \leq n-1$ and $n \geq N$.
In the same way, let $\mathcal{L}_{up}^*(\psi,N) \subseteq \mathcal{L}$ denote the set of mappings $\sigma \in \mathcal{L}$, which satisfy
\begin{align*}
\sigma(i,n-i)+\sigma(n-i,i)\leq \psi(n)
\end{align*}
for every integer $1 \leq i \leq n-1$ and $n \geq N$.
\end{defi}
Note that we have $\mathcal{L}_{up}^*(\psi,N)\subseteq \mathcal{L}_{up}(\psi,N) \subseteq \mathcal{L}_{up}^*(2\psi,N)$. Furthermore, note that without loss of generality, if we consider the class $\mathcal{L}_{up}(\psi,N)$, we assume that $\psi(n) \geq 1/(n-1)$ for every integer $n \geq N$, as the values $\sigma(i,n-i)$ have to add up to $1$ for $1 \leq i \leq n-1$ by condition \eqref{sigmacond}. 
In the same way, if we consider the class $\mathcal{L}_{up}^*(\psi,N)$, we can assume that $\psi(n) \geq 2/(n-1)$ for every integer $n \geq N$.
As our first main theorem, we show the following bound on $\mathbb{E}(H_{n,\sigma})$ for the class of $\psi$-upper-bounded sources.

\begin{theorem}\label{thm:upper-bounded}
Let $\sigma \in \mathcal{L}_{up}^*(\psi,N)$. 
If there is a differentiable, non-negative function $\vartheta \colon \mathbb{R}^+ \to \mathbb{R}$ with non-constant derivative $\vartheta'$, such that 
\begin{itemize}
\item[(i)] $\vartheta'$ is increasing for $x\geq 1$, and
\item[(ii)] $e\psi(x)\vartheta(x) \leq \vartheta'(x)$ for all $x \geq N$,
\item[(iii)] $\vartheta'(1)\geq 1$,
\end{itemize}
then 
\begin{align*}
\mathbb{E}(H_{n,\sigma})\leq \ln \vartheta'(n)(1+o(1)).
\end{align*}
\end{theorem}
\begin{proof}
In order to prove the statement, we consider the expected value
\begin{align*}
\mathbb{E}(e^{H_{n,\sigma}})=\sum_{t \in \mathcal{T}_n}P_{\sigma}(t)e^{h(t)}
\end{align*}
and show that 
\begin{align*}
\mathbb{E}(e^{H_{n,\sigma}})\leq e^N\vartheta'(n).
\end{align*}
We prove the statement inductively in $n$.
For the base case, we consider the integers $1 \leq n < N$.
By assumption (iii), we have $\vartheta'(1)\geq 1$ and thus $\vartheta'(x)\geq 1$ for all $x \geq 1$.
As the maximal height of a binary tree of size $n < N$ is bounded from above by $N$, we find
$
\mathbb{E}(e^{H_{n,\sigma}})\leq e^N\leq \vartheta'(n)e^N
$
for every integer $1 \leq n <N$.
Hence, the base case follows. 
For the induction step, let $n \geq N$. 
Applying Lemma~\ref{lemma:recursive} (with $\varphi_n=e$ for all $n \geq 0$), we find
\begin{align*}
\mathbb{E}(e^{H_{n,\sigma}})&\leq e\sum_{k=1}^{n-1}\sigma(k,n-k)\left(\mathbb{E}(e^{H_{k,\sigma}})+\mathbb{E}(e^{H_{n-k,\sigma}})\right)\\
&=e\sum_{k=1}^{n-1}\left(\sigma(k,n-k)+\sigma(n-k,k)\right)\mathbb{E}(e^{H_{k,\sigma}})\\
&\leq e^{N+1}\sum_{k=1}^{n-1}\left(\sigma(k,n-k)+\sigma(n-k,k)\right)\vartheta'(k),
\end{align*}
where the last inequality follows from the induction hypothesis.
As $\sigma \in \mathcal{L}^*_{up}(\psi)$, we obtain
\begin{align*}
\mathbb{E}(e^{H_{n,\sigma}})\leq e^{N+1} \psi(n)\sum_{k=1}^{n-1}\vartheta'(k).
\end{align*}
As $\vartheta'$ is increasing by condition (i), we find
\begin{align*}
\mathbb{E}(e^{H_{n,\sigma}})\leq e^{N+1} \psi(n)\int_{k=1}^{n}\vartheta'(x) \mathrm{dx} = e^{N+1} \psi(n)\left(\vartheta(n)-\vartheta(1)\right) \leq e^{N+1} \psi(n) \vartheta(n),
\end{align*}
since $\vartheta(1)\geq 0$ by assumption.
By condition (ii), we now have $e \psi(n)\vartheta(n) \leq \vartheta'(n)$, which concludes the induction.
Using Jensen's inequality, we now obtain
\begin{align*}
\mathbb{E}(H_{n,\sigma}) \leq \ln \left(\vartheta'(n)\right)+N,
\end{align*}
which concludes the proof.
\end{proof}

From Theorem~\ref{thm:upper-bounded}, we obtain the following corollary.

\begin{cor}\label{cor:xhochalpha}
Let $\sigma \in \mathcal{L}_{up}^*(\psi,N)$, where $N \in \mathbb{N}$, $\psi(x)=c\cdot x^{-\alpha}$ for $0 \leq \alpha \leq 1$ and assume that $c\geq  e^{-1}$ holds\footnote{Recall that if we consider the class $\mathcal{L}_{up}^*(\psi)$, the mapping $\psi$ has to satisfy the lower bound $\psi(x)\geq 2/(x-1)$, such that $ec-1>0$ always holds in the case $\alpha=1$ in Corollary~\ref{cor:xhochalpha}.}. Then
\begin{align*}
\mathbb{E}(H_{n,\sigma}) \leq \begin{dcases}
\frac{ec}{1-\alpha}n^{1-\alpha}(1+o(1)), \quad &\text{ if } 0 \leq \alpha < 1,\\
(ec-1)\ln(n)(1+o(1)), &\text{ if } \alpha=1.
\end{dcases}
\end{align*}
\end{cor}
\begin{proof}
For the proof, we apply Theorem~\ref{thm:upper-bounded}.
Thus, we have to find a suitable function $\vartheta$, which satisfies the requirements of Theorem~\ref{thm:upper-bounded}.
Define
\begin{align*}
\xi(x)=\begin{dcases}
\frac{c x^{1-\alpha}}{1-\alpha} \quad &\text{ if } 0 \leq \alpha < 1,\\
c \ln x \quad &\text{ if } \alpha = 1,\end{dcases} 
\end{align*}
and set $\vartheta(x)=e^{e\xi(x)}$.
We find that $\vartheta\colon \mathbb{R}^+\to \mathbb{R}$ is differentiable and non-negative. We have
\begin{align*}
\vartheta'(x)=\begin{dcases}
ce^{1+\frac{ec}{1-\alpha}x^{1-\alpha}}x^{-\alpha} &\text{ if } 0 \leq \alpha < 1, \\
ecx^{ec-1} &\text{ if } \alpha = 1.
\end{dcases}
\end{align*}

Clearly, in the case that $\alpha=1$, we find that $\vartheta'$ is increasing. 
In the case that $0 \leq \alpha <1$, we consider the function $\vartheta''$:
We have
\begin{align*}
\vartheta''(x)=ce^{1+\frac{ce}{1-\alpha}x^{1-\alpha}}\left(ecx^{-2\alpha}-\alpha x^{-\alpha-1}\right)
\end{align*}
in this case. As $\frac{ec}{\alpha}\geq 1 \geq x^{\alpha-1}$ for all $x\geq 1$, we find that $\vartheta''(x)\geq 0$ for all $x\geq 1$, such that $\vartheta'$ is increasing for $x\geq 1$ and condition (i) of Theorem~\ref{thm:upper-bounded} is satisfied.
Furthermore, a short calculation shows that condition (ii) is satisfied (in fact, equality holds in condition (ii)).
Finally, as $c \geq e^{-1}$ by assumption, we find that $\vartheta'(1)=ce^{1+ec/(1-\alpha)}\geq 1$, respectively, $\vartheta'(1)=ec \geq 1$, such that condition (iii) is also satisfied.
The statement of the corollary now follows from Theorem~\ref{thm:upper-bounded}.
\end{proof}
We remark that Theorem~\ref{thm:upper-bounded} can be applied to classes $\mathcal{L}_{up}^*(\psi)$ for many more functions $\psi$. Moreover, recall that via the inclusion $\mathcal{L}_{up}(\psi, N) \subseteq \mathcal{L}^*_{up}(2\psi,N)$, we immediately obtain results on $\mathbb{E}(H_{n,\sigma})$ for $\sigma \in \mathcal{L}_{up}(\psi, N)$ from Theorem~\ref{thm:upper-bounded} for some functions $\psi$ as well.
For the binary search tree model (Example~\ref{ex:bst}), the uniform model (Example~\ref{ex:uniform}) and the binomial random tree model (Example~\ref{ex:dst}), Theorem~\ref{thm:upper-bounded} yields the following asymptotic upper bounds on the average height:

\begin{example}\label{ex:upper-bounded-bst}
For the binary search tree model, we find that $\sigma_{bst} \in \mathcal{L}_{up}^*(\psi_{bst},N)$, where $\psi_{bst}\colon \mathbb{R} \to \mathbb{R}$ is given by $\psi_{bst}=2/(x-1)$ and $N \geq 2$.
Choose
\begin{align*}
\vartheta_{bst}(x)= x^{2e\frac{N}{N-1}},
\end{align*}
for any $N \geq 2$, then $\vartheta_{bst}$ clearly satisfies conditions (i) and (iii) of Theorem~\ref{thm:upper-bounded}. 
Moreover, we have
\begin{align*}
\frac{2e}{x-1}x^{2e\frac{N}{N-1}}\leq 2e\frac{N}{N-1}x^{2e\frac{N}{N-1}-1},
\end{align*}
for $x \geq N$, such that condition (ii) is also satisfied.
Hence, we have
\begin{align*}
\mathbb{E}(H_{n,\sigma_{bst}})\leq \bigl(2e\frac{N}{N-1}-1\bigr)\ln(n)(1+o(1)),
\end{align*}
where the choice of $N\geq 2$ is arbitrary.
Thus, Theorem~\ref{thm:upper-bounded} yields for the binary search tree model $\mathbb{E}(H_{n,\sigma_{bst}})\leq c\ln(n)(1+o(1))$ for any constant $c>2e-1\approx 4.436564$. 
Recall that it was shown e.g. in~\cite{Devroye86}, that $\mathbb{E}(H_{n,\sigma_{bst}})\leq c_{bst}\ln(n)(1+o(1))$, where $c_{bst} \approx 4.31107...$ is the unique solution of $x\ln(2e/x)=1$ with $x \geq 2$. 
Thus, our general Theorem~\ref{thm:upper-bounded} yields a result quite close to the precise asymptotic growth of the average height in the binary search tree model, while at the same time covering a whole class of sources.
\end{example}

\begin{example}\label{ex:upper-bounded-dst}
For the binomial random tree model, it was shown in \cite{SeelbachLohreyWagner18}, that $\sigma_{bin,p} \in \mathcal{L}(\psi_{bin,p},N)$ for a function $\psi_{bin,p}(x) \in \Theta(x^{-1/2})$ and an integer $N \geq 2$.
Thus, there is $c>e^{-1}$ and $N' \geq N$, such that
$\sigma_{bin,p} \in \mathcal{L}(\psi,N')$ for $\psi(x)=c\cdot x^{-1/2}$.
By Corollary~\ref{cor:xhochalpha}, we find that
\begin{align*}
\mathbb{E}(H_{n,\sigma_{bin,p}})\in \mathcal{O}(n^{1/2}).
\end{align*}
In the next section, we show that this result can be improved: In Example~\ref{ex:dst-weakly}, we obtain an upper bound $\mathbb{E}(H_{n,\sigma_{bin,p}})\in \mathcal{O}(\log n)$ for the binomial random tree model.
\end{example}

Finally, for the uniform distribution, Theorem~\ref{thm:upper-bounded} and Corollary~\ref{cor:xhochalpha} only yield a trivial upper bound of the form $\mathbb{E}(H_{n,\sigma_{uni}}) \in \mathcal{O}(n)$, as we have $\sigma_{uni}(1,n-1)\geq 1/4$ for every $n \geq 2$.
However, also for the uniform distribution, we will show a non-trivial upper bound in the next section (Example~\ref{ex:uniform-weakly}).

\subsection{Weakly-balanced sources}

In this section, we investigate another class of leaf-centric binary tree sources, which is called \emph{weakly-balanced} binary tree sources.
This class has also been considered in the context of DAG-compression in \cite{SeelbachLohreyWagner18}, and represents a generalization of balanced binary tree sources introduced in \cite{ZhangYK14} and further analysed in \cite{GanardiHuckeLohreySeelbach19}.

\begin{defi}[$\phi$-weakly-balanced sources]
Let $\phi\colon \mathbb{N} \to (0,1]$ denote a decreasing function, let $\gamma \in (0,\frac{1}{2})$ and let $N \in \mathbb{N}$.
With $\mathcal{L}_{wbal}(\phi, \gamma,N)\subseteq \mathcal{L}$ we denote the set of mappings $\sigma \in \mathcal{L}$, which satisfy 
\begin{align*}
\sum_{\gamma n \leq k \leq (1-\gamma)n}\sigma(k,n-k)\geq \phi(n)
\end{align*}
for every $n \geq N$. 
\end{defi}

For the class of weakly-balanced tree sources, we show our second main theorem:
\begin{theorem}\label{thm:weakly-balanced-upper-bound}
Let $\sigma \in \mathcal{L}_{wbal}(\phi,\gamma, N)$. Then
\begin{align*}
\mathbb{E}(H_{n,\sigma}) \leq \frac{1}{\log(\frac{1}{1-\gamma})}\cdot \frac{\log(2\phi(n)^{-1}(1+\phi(n)))}{\log(1+\phi(n))}\cdot \log n(1+o(1)).
\end{align*}
\end{theorem}

\begin{proof}

We set $\varphi_n = 1+\phi(n)$ and 
\begin{align*}
\kappa_n = \frac{\log(2\phi(n)^{-1}\varphi_n)}{\log((1-\gamma)^{-1})}.
\end{align*}
Note that as $\phi$ is a decreasing function, we find that the sequence $(\kappa_n)_{n \in \mathbb{N}}$ is increasing in $n$ and the sequence $(\varphi_n)_{n \in \mathbb{N}}$ is decreasing in $n$.
In order to prove the statement, we consider the expected value
\begin{align*}
\mathbb{E}(\varphi_n^{H_{n,\sigma}})=\sum_{t \in \mathcal{T}_n}P_{\sigma}(t)\varphi_n^{h(t)}
\end{align*}
and show that
\begin{align}\label{eq:induction-weakly-balanced}
\mathbb{E}(\varphi_n^{H_{n,\sigma}})\leq 2^Nn^{\kappa_n}.
\end{align}
We prove the statement \eqref{eq:induction-weakly-balanced} inductively in $n$.
For the base case, we consider the integers $1 \leq n < N$.
The maximal height of a binary tree of size $n < N$ is bounded from above by $N$.
We thus have $\mathbb{E}(\varphi_n^{H_{n,\sigma}})\leq \varphi_n^N\leq \varphi_n^Nn^{\kappa_n}\leq 2^Nn^{\kappa_n}$ for every integer $1 \leq n < N$.
\\

For the induction step, let $n \geq N$. By Lemma~\ref{lemma:recursive}, we find that $\mathbb{E}(\varphi_n^{H_{n,\sigma}})$ satisfies the following recursive estimate:
\begin{align*}
\mathbb{E}(\varphi_n^{H_{n,\sigma}})
\leq \varphi_n\sum_{k=1}^{n-1}\sigma(k,n-k)\left(\mathbb{E}(\varphi_k^{H_{k,\sigma}})+\mathbb{E}(\varphi_{n-k}^{H_{n-k,\sigma}})\right).
\end{align*}
By induction hypothesis, we have
\begin{align}\label{eq:weakly-balanced-induction-step}
\mathbb{E}(\varphi_n^{H_{n,\sigma}})\leq 2^N\varphi_n \sum_{k=1}^{n-1}\sigma(k,n-k)\left(k^{\kappa_k}+(n-k)^{\kappa_{n-k}}\right).
\end{align}
We distinguish two cases.\\

\textit{Case $1$:} $n \leq \gamma^{-1}$ (and thus, $\gamma n \leq 1$).
Reordering the sum on the right-hand side of \eqref{eq:weakly-balanced-induction-step}, we obtain
\begin{align*}
\mathbb{E}(\varphi_n^{H_{n,\sigma}})&\leq 2^N\varphi_n \sum_{k=1}^{n-1}\left(\sigma(k,n-k)+\sigma(n-k,k)\right)k^{\kappa_k}.
\end{align*}
As $(\kappa_n)_{n \in \mathbb{N}}$ is increasing, and as $\sigma \in \mathcal{L}$, we have
\begin{align*}
\mathbb{E}(\varphi_n^{H_{n,\sigma}})&\leq 2^{N+1}\varphi_n (n-1)^{\kappa_{n-1}} \leq 2^{N+1}\varphi_n (n-1)^{\kappa_{n}}.
\end{align*} 
By definition of $\kappa_n$, we have
\begin{align*}
\kappa_n = \frac{\log(2\phi(n)^{-1}\varphi_n)}{\log((1-\gamma)^{-1})}\geq \frac{\log(2\varphi_n)}{\log(n(n-1)^{-1})},
\end{align*}
for every integer $1 \leq n \leq \gamma^{-1}$. Thus, we have
$2\varphi_n(n-1)^{\kappa_n}n^{-\kappa_n} \leq 1$ for every integer $1 \leq n \leq \gamma^{-1}$, which is equivalent to $2\varphi_n (n-1)^{\kappa_{n}}\leq n^{\kappa_n}$. The statement for case $1$ hence follows.\\

\textit{Case $2$:} $n > \gamma^{-1}$ (and hence, $\gamma n > 1$).
We obtain from inequality \eqref{eq:weakly-balanced-induction-step}:
\begin{align*}
\mathbb{E}(\varphi_n^{H_{n,\sigma}})\leq \ &2^N\varphi_n \sum_{1 \leq k < \gamma n}\sigma(k,n-k)\left(k^{\kappa_n}+(n-k)^{\kappa_{n}}\right)\\
+\ &2^N\varphi_n \sum_{\gamma n \leq k \leq (1-\gamma)n}\left(\sigma(k,n-k)+\sigma(n-k,k)\right)k^{\kappa_n}\\
+\ &2^N\varphi_n \sum_{(1-\gamma)n < k \leq n-1}\sigma(k,n-k)\left(k^{\kappa_n}+(n-k)^{\kappa_{n}}\right).
\end{align*}
With $(n-k)^{\kappa_n}+k^{\kappa_n}\leq n^{\kappa_n}$ for every integer $1 \leq k \leq n-1$, we obtain
\begin{align*}
\mathbb{E}(\varphi_n^{H_{n,\sigma}})\leq & \  2^N\varphi_n n^{\kappa_n}\Bigl(\sum_{1 \leq k < \gamma n}\sigma(k,n-k) + \sum_{(1-\gamma)n < k \leq n-1}\sigma(k,n-k)\Bigr)\\
&+2^{N+1}\varphi_n(1-\gamma)^{\kappa_n}n^{\kappa_n} \sum_{\gamma n \leq k \leq (1-\gamma)n}\sigma(k,n-k).
\end{align*}
We set
\begin{align*}
\alpha:=\sum_{\gamma n \leq k \leq (1-\gamma)n}\sigma(k,n-k)\
\end{align*}
and obtain 
\begin{align}\label{eq:weakly-balanced-linear}
\mathbb{E}(\varphi_n^{H_{n,\sigma}})\leq 2^N\left( (1-\alpha) \varphi_n n^{\kappa_n}+2\alpha \varphi_n(1-\gamma)^{\kappa_n}n^{\kappa_n}\right).
\end{align}
The right-hand side of \eqref{eq:weakly-balanced-linear} is either linearly increasing or decreasing in $\alpha$. As $\sigma \in \mathcal{L}_{wbal}(\phi, \gamma, N)$ and $n \geq N$, we have $\phi(n) \leq \alpha \leq 1$ and thus, the right-hand side of~\eqref{eq:weakly-balanced-linear} attains its maximal value either at $\alpha=1$ or $\alpha=\phi(n)$.
For $\alpha=1$, we obtain
\begin{align*}
\mathbb{E}(\varphi_n^{H_{n,\sigma}})\leq 2^{N+1}\varphi_n(1-\gamma)^{\kappa_n}n^{\kappa_n}.
\end{align*}
As $\kappa_n = \log(2\phi(n)^{-1}\varphi_n)(\log((1-\gamma)^{-1}))^{-1}$, we find $2\varphi_n(1-\gamma)^{\kappa_n}n^{\kappa_n}\leq \phi(n)n^{\kappa_n}$ and hence, $\mathbb{E}(\varphi_n^{H_{n,\sigma}})\leq 2^Nn^{\kappa_n}$ in this case.
 For $\alpha=\phi(n)$, we have
 \begin{align*}
\mathbb{E}(\varphi_n^{H_{n,\sigma}})\leq 2^N\left( (1-\phi(n)) \varphi_n n^{\kappa_n}+2\phi(n) \varphi_n(1-\gamma)^{\kappa_n}n^{\kappa_n} \right).
\end{align*}
Again by definition of $\kappa_n$, we have $2\phi(n) \varphi_n(1-\gamma)^{\kappa_n}n^{\kappa_n}\leq \phi(n)^2n^{\kappa_n} $.
By definition of $\varphi_n$, we thus have
  \begin{align*}
\mathbb{E}(\varphi_n^{H_{n,\sigma}})\leq 2^N \left( (1-\phi(n)) (1+\phi(n)) n^{\kappa_n}+\phi(n)^2n^{\kappa_n}\right) = 2^N n^{\kappa_n}.
\end{align*}
This concludes the induction.
Using Jensen's inequality, we now obtain from estimate~\eqref{eq:induction-weakly-balanced}:
\begin{align*}
\mathbb{E}(H_{n,\sigma})\leq \frac{\kappa_n\log n + N}{\log(1+\phi(n))}.
\end{align*}
The statement of the theorem follows.
\end{proof}

In the following examples, we apply Theorem~\ref{thm:weakly-balanced-upper-bound} to the random tree models from Example~\ref{ex:bst}, Example~\ref{ex:uniform} and Example~\ref{ex:dst}.

\begin{example}\label{ex:bst-weakly}
For the binary search tree model, we find that $\sigma_{bst}\in\mathcal{L}_{wbal}(\phi_{bst}, \gamma_{bst}, N)$ with $\phi_{bst}(n)=1/2$ for every $n \geq 2$ and $\gamma_{bst}=1/4$. 
Hence, by Theorem~\ref{thm:weakly-balanced-upper-bound}, we obtain
\begin{align*}
\mathbb{E}(H_{n, \sigma_{bst}}) \leq \frac{\log(6)}{\log(4/3)\log(3/2)}\log n (1+o(1)),
\end{align*}
where the leading constant evaluates to $\log(6)\log(4/3)^{-1}\log(3/2)^{-1} \approx 10.65$.
Recall that in Example~\ref{ex:upper-bounded-bst}, we have obtained a stronger upper bound on $\mathbb{E}(H_{n, \sigma_{bst}})$.
\end{example}

\begin{example}\label{ex:dst-weakly}
For the binomial random model $\sigma_{bin,p}$, it is shown in \cite[Example 3.11]{SeelbachLohreyWagner18}, that for every $0 < \varepsilon < 1$ and $\gamma < \min\{p,1-p\}$, there is an integer $N$ (depending on $p$, $\gamma$ and $\varepsilon$), such that $\sigma_{bin,p} \in \mathcal{L}_{wbal}(\phi,\gamma, N)$, where $\phi$ is the function defined by $\phi(n)=1-\varepsilon$ for every $n \geq 0$. 
Thus, Theorem~\ref{thm:weakly-balanced-upper-bound} yields
\begin{align*}
\mathbb{E}(H_{n,\sigma_{bin,p}})\leq c_p\cdot \log(n)(1+o(1)),
\end{align*}
where we can choose any constant $c_p > 2\cdot \max\{\log(1/p)^{-1}, \log(1/(1-p))^{-1}\}$, as $\varepsilon>0$ is arbitrary.
In particular, Theorem~\ref{thm:weakly-balanced-upper-bound} yields a more precise upper bound on $\mathbb{E}(H_{n,\sigma_{bin,p}})$ than Theorem~\ref{thm:upper-bounded} (see Example~\ref{ex:upper-bounded-dst}).
\end{example}

\begin{example}\label{ex:uniform-weakly}
Let us consider the uniform distribution $\sigma_{uni}$ next. 
In \cite[Example 3.15]{SeelbachLohreyWagner18}, it is shown that $\sigma_{uni}$ belongs to a particular class of leaf-centric tree sources, which is introduced in \cite[Definition 3.12]{SeelbachLohreyWagner18}:
Let $\vartheta\colon \mathbb{R} \to (0,1]$ be a decreasing function, let $N \in \mathbb{N}$ and let $0<\gamma < 1/2$. The class of \emph{$\vartheta$-strongly-balanced sources} $\mathcal{L}_{sbal}(\vartheta, \gamma,N) \subseteq \mathcal{L}$ is defined as the set of mappings $\sigma$, for which there is a constant $c\geq 1$, such that for every $n \geq N$ and every integer $r$ with $c \leq r \leq \lceil \gamma n\rceil$, the following inequality holds:
\begin{align*}
\sum_{r \leq i \leq n-r} \sigma(i,n-i)\geq \vartheta(r).
\end{align*}
In particular, every strongly balanced leaf-centric tree source is also weakly-balanced: Let $\sigma \in \mathcal{L}_{sbal}(\vartheta, \gamma,N)$ and define $\varphi\colon \mathbb{N} \to (0,1]$ by $\phi(n)=\vartheta(\lceil \gamma n\rceil)$. Then we find that $\sigma \in \mathcal{L}_{wbal}(\phi, \gamma, N)$.
Thus, we can apply \cite[Example 3.15]{SeelbachLohreyWagner18} in order to show that $\sigma_{uni}$ is weakly-balanced.
From \cite[Example 3.15]{SeelbachLohreyWagner18}, we find that for every $\delta>1$ and $0 < \gamma < 1/2$, there is an integer $N$, such that $\sigma_{uni} \in \mathcal{L}_{wbal}(\phi_{uni}, \gamma, N)$, where $\phi_{uni}\colon \mathbb{N} \to (0,1]$ is defined by
\begin{align*}
\phi_{uni}(n)=\frac{1-2\gamma}{\delta\sqrt{\pi \gamma n}}.
\end{align*}
Thus, we can apply Theorem~\ref{thm:weakly-balanced-upper-bound}. With
$
\log(2\phi_{uni}(n)^{-1}(1+\phi_{uni}(n))) \in \mathcal{O}(\log n)
$
and $\log(1+\phi_{uni}(n))^{-1} \in \mathcal{O}(\sqrt{n})$, we obtain
\begin{align*}
\mathbb{E}(H_{n,\sigma_{uni}}) \in \mathcal{O}\left(\sqrt{n} \log (n)^2\right).
\end{align*}
We remark that in~\cite{FlajoletOdlyzko82}, it is shown that $\mathbb{E}(H_{n,\sigma_{uni}}) \in \Theta(\sqrt{n})$. 
Hence, our upper bound is slightly weaker than the correct asymptotic growth of $\mathbb{E}(H_{n,\sigma_{uni}})$.
\end{example}

\section{Conclusion and open problems}

We have derived upper bounds on the average height of random binary trees with respect to two classes of leaf-centric binary tree sources.

With respect to the uniform probability distribution (Example~\ref{ex:uniform}), our main results only yield a trivial upper bound (Theorem~\ref{thm:upper-bounded}), respectively, an upper bound that does not match the correct asymptotic growth of the average height in the uniform model by a factor of $\log(n)^2$ (Theorem~\ref{thm:weakly-balanced-upper-bound} and Example~\ref{ex:uniform-weakly}). A natural open question is thus to find a non-trivial class of leaf-centric binary tree sources that contains the uniform model, 
such that an upper bound on the average height with respect to this whole class yields the correct asymptotic upper bound on the average height for the uniform model.
A possible candidate for this is the class of strongly-balanced leaf-centric binary tree sources, which was introduced in \cite{SeelbachLohreyWagner18} (see also Example~\ref{ex:uniform-weakly}). In particular, whereas the results from \cite{SeelbachLohreyWagner18} on the average DAG-size for upper-bounded and weakly-balanced leaf-centric binary tree sources only yield trivial upper bounds with respect to the uniform distribution, the result from~\cite{SeelbachLohreyWagner18} for strongly-balanced leaf-centric tree sources gives the correct asymptotic bound for the average DAG-size in the uniform model.
This suggests that a possible result on the average height for the class of strongly-balanced leaf-centric tree sources could yield a stronger upper bound on the average height in the uniform model.

Furthermore, another topic for future research might be to derive \emph{lower bounds} on the average height of random binary trees with respect to classes of leaf-centric tree sources. 
A possible class with respect to which a lower bound on the average height could be derived is the class of \emph{unbalanced} leaf-centric binary tree sources, which was introduced in \cite{SeelbachLohreyWagner18} (again in the context of DAG-compression).

Moreover, leaf-centric binary tree sources have been generalized to  random tree models for plane trees called \emph{fixed-size ordinal tree sources} in \cite{MunroNSW21}. Future research might be to investigate the average height of plane trees with respect to classes of fixed-size ordinal tree sources.\\

\noindent
\textbf{Acknowledgements.}
The authors thanks Markus Lohrey for valuable discussions
and helpful comments on this paper.

\bibliographystyle{plain}
\bibliography{bib}

\end{document}